\documentclass[12]{amsart}
\usepackage{amssymb} 
\newtheorem{thm}{Theorem}[section]
\newtheorem{cor}[thm]{Corollary}
\newtheorem{lem}[thm]{Lemma}

\newtheorem{rem}[thm]{Remark}

\theoremstyle{question}
\newtheorem{qu}[thm]{Question}

\numberwithin{equation}{section} 

\newcommand{\Z}{\mathbb{Z} }
\newcommand{\N}{\mathbb{N} }

\begin{document}
\author[Abdollahi and Khosravi]{A. Abdollahi  \;\;\&\;\; H. Khosravi}
\title[Right and left Engel element]{On the right and left $4$-Engel elements}
\address{$^\mathbf{1}$Department of Mathematics, University of Isfahan, Isfahan 81746-73441, Iran}%
\address{$^\mathbf{2}$School of Mathematics, Institute for Research in Fundamental Sciences (IPM), P.O.Box: 19395-5746, Tehran, Iran.}%
\email{($^\mathbf{1,2}$A. Abdollahi) \;\; a.abdollahi@math.ui.ac.ir \;\ abdollahi@member.ams.org}%
\email{hassan\_khosravy@yahoo.com}
\thanks{The first author's research was in part supported by a grant from IPM (No. 87200118).}
\begin{abstract}
In this paper we study left and right 4-Engel elements of a group.
In particular, we prove that $\langle a, a^b\rangle$ is nilpotent
of class at most 4, whenever $a$ is any element and $b^{\pm
1}$ are right 4-Engel elements or $a^{\pm 1}$ are left 4-Engel
elements and $b$ is an arbitrary element of $G$. Furthermore we
prove that for any prime $p$ and any element $a$ of finite
$p$-power order in a group $G$ such that $a^{\pm 1}\in L_4(G)$,
$a^4$, if $p=2$, and $a^p$, if $p$ is an odd prime number, is in
the Baer radical of $G$.
\end{abstract}
\maketitle
\section{\bf Introduction and Results}
Let $G$ be any group and $n$ be a non-negative integer. For any two
 elements $a$ and $b$ of $G$, we define
inductively $[a,_n b]$, the $n$-Engel commutator of the pair
$(a,b)$, as follows:
$$[a,_0 b]:=a,~ [a,b]=[a,_1 b]:=a^{-1}b^{-1}ab \mbox{ and }[a,_n
b]=[[a,_{n-1} b],b]\mbox{ for all }n>0.$$ An element $x$ of $G$ is
called right $n$-Engel if $[x,_ng]=1$ for all $g\in G$. We denote
by $R_n(G)$, the set of all right $n$-Engel elements of $G$. The
corresponding subset to $R_n(G)$ which can be similarly defined is
$L_n(G)$ the set of all left $n$-Engel elements of $G$ where an
element $x$ of $G$ is called  left $n$-Engel element if $[g,_n
x]=1$ for all $g\in G$. A group $G$ is called $n$-Engel if
$G=L_n(G)$ or equivalently $G=R_n(G)$. It is clear that
$R_0(G)=1$, $R_1(G)=Z(G)$ the center of $G$, and by a result of
Kappe \cite{kappe}, $R_2(G)$ is a
characteristic subgroup of $G$. Also we have $L_0(G)=1$,
$L_1(G)=Z(G)$ and it can be easily seen that
$$L_2(G)=\{x\in G\;|\; \langle x\rangle^G ~ \text{is abelian} \},$$
where $\langle x\rangle^G$ denotes the normal closure of $x$ in
$G$. In \cite{ab1} it is shown that
$$L_3(G)=\{x\in G \;|\; \langle x,x^y\rangle \in \mathcal{N}_2 ~\text{ for all }~y\in G\},$$
where $\mathcal{N}_2$ is the class of nilpotent groups of class at
most 2. Also it is proved that $\langle x,y\rangle$ is nilpotent
of class at most 4 whenever $x,y\in L_3(G)$.
Newell \cite{newell} has shown that the normal closure of every
element of $R_3(G)$ in $G$ is nilpotent of class at most 3. This shows that $R_3(G)\subseteq Fit(G)$, where $Fit(G)$ is the
Fitting subgroup of $G$, and in particular it is contained in
$B(G)\subseteq HP(G)$ where $B(G)$ and $HP(G)$ are the Baer
radical and Hirsch-Plotkin radical of $G$, respectively. It is
clear that
$$R_0(G)\subseteq R_1(G)\subseteq R_2(G)\subseteq\cdots \subseteq
R_n(G)\subseteq\cdots. $$ Gupta and Levin \cite{gupta} have  shown that
the normal closure of an element in a 5-Engel group need not be
nilpotent (see also \cite{vaghan} p. 342). Therefore
$R_n(G)\nsubseteq Fit(G)$ for $n\geq 5$. The following question naturally
arises:
\begin{qu}\label{q1} Let $G$ be an arbitrary group. What are the least positive
integers $n$, $m$ and $p$ such that $R_n(G)\nsubseteq Fit(G)$, $R_m(G)\nsubseteq B(G)$ and $R_p(G)\nsubseteq HP(G)$?
\end{qu}
To find integer $n$ in   Question \ref{q1} we have to answer the
following.
\begin{qu} Let $G$ be an arbitrary group. Is it true that
$R_4(G)\subseteq Fit(G)$?
\end{qu}
 Although in \cite{ab1} it is shown
that there exists $n \in \N$ such that $L_n(G)\nsubseteq HP(G)$,
the following question   is still open.
\begin{qu} Let $G$ be an arbitrary group. What is the least positive
integer $k$  such that $L_k(G)\nsubseteq HP(G)$?
\end{qu}
In this
paper we study right and left 4-Engel elements. Our main results are the following.
\begin{thm}\label{th3}
Let $G$ be any group. If $a\in G$  and $b^{\pm 1}\in R_4(G)$, then
$\langle a,a^b\rangle$ is nilpotent of class at most $4$.
\end{thm}
\begin{thm}\label{th4}
Let $G$ be an arbitrary group and $a^{\pm 1}\in L_4(G)$. Then  $\langle a,a^b\rangle$ is nilpotent of class at most
$4$ for all $b\in G$.
\end{thm}
\begin{thm}\label{th5}
Let $G$ be a group and  $a^{\pm 1}\in L_4(G)$ be a $p$-element
for some prime $p$. Then
\begin{enumerate}
    \item If $p=2$ then $a^4\in B(G)$.
    \item If $p$ is an odd prime, then $a^p\in B(G)$.
\end{enumerate}
\end{thm}
In \cite{traus,traus2} Traustason has proved the  above results for a 4-Engel group $G$ (in which all elements are
simultaneously right and left 4-Engel). The proofs of Theorems \ref{th3}, \ref{th4} and \ref{th5} are somehow
inspired by the arguments of \cite{traus,traus2}.

In Section 4, we will show that in Theorem \ref{th3}, we cannot remove the condition $b^{-1}\in
R_4(G)$ and  that in  Theorems \ref{th4}
and \ref{th5}, the condition $a^{-1}\in L_4(G)$ is an necessary condition.

Macdonald \cite{macd} has shown that the inverse or square of a
right 3-Engel element need not be right 3-Engel. In Section 4   the GAP
\cite{gap} implementation of Werner Nickel's nilpotent quotient
algorithm \cite{nick} is used  to prove that the inverse of a right (left, respectively)
4-Engel element is not necessarily  a right (left, respectively) 4-Engel element.

\section{\bf Right $4$-Engel elements}
In this section we prove Theorem \ref{th3}.
\begin{lem}\label{car}
Let G be a group with elements $a, b$. Let $x=a^b$. We have
\begin{enumerate}
\item[(a)] $[b^{-1}, a, a, a, a]=1$ if and only if $[x^{-1}, x^a]$ commutes with $x^a$.
\item[(b)]  $[b^{-1}, a^{\tau}, a^{\tau}, a^{\tau},a^{\tau}]=1$ for $\tau=1$ and $\tau=-1$ if and only if $\langle x^a,x\rangle$  is nilpotent of class at
most $2$.
\item[(c)] $[b^{\epsilon}, a^{\tau}, a^{\tau},  a^{\tau}, a^{\tau} ] = 1$ for all $\epsilon,\tau\in\{-1,+1\}$ if and only if
$\langle a,[a^b,a]\rangle$ and $\langle a^b, [a^b,a]\rangle$ are nilpotent of class at most $2$.
\end{enumerate}
\end{lem}
\begin{proof}
We use  a general trick for
Engel commutators,  namely $$[b^{-1},_n a]=1 \Leftrightarrow [a^{-1},_{n-1} x] = 1 \;\; \text{where} \; x = a^b.$$
Applying this trick twice gives $$[b^{-1},_4 a]=1 \Leftrightarrow  [a^{-1},_3 x] = 1 \Leftrightarrow [x^{-1},_2 x^a] = 1.$$
This proves (a). To prove (b) note that we also have
$$[b^{-1},_4 a^{-1}]=1  \Leftrightarrow [a,_3 x^{-1}] = 1 \Leftrightarrow [x,_2 x^{-a^{-1}}]=1 \Leftrightarrow [x^{a},_2 x^{-1}]=1.$$
To show (c), it follows from (b) that  $\langle x,x^a\rangle=\langle a^b,(a^b)^a\rangle=\langle
a^b,[a^b,a]\rangle$ is nilpotent of class at most 2. Also we have $\langle a^{b^{-1}},[a^{b^{-1}},a]\rangle$ is nilpotent of class at most $2$.
Now the conjugate of the latter subgroup by $b$, it follows that $\langle a,[a^b,a]\rangle$. This completes the proof.
\end{proof}
As  immediate corollaries it follows then that
\begin{cor}\label{co1}
\begin{enumerate}
\item[(a)] $b^{\epsilon}\in R_4(G)$ for both $\epsilon=1$ and $\epsilon=-1$ if and only if $\langle a, [a^b, a]\rangle$ and $\langle a^b, [a^b, a]\rangle$ are nilpotent of class at most $2$ for all $a\in G$.
\item[(b)]    $a^{\epsilon}\in L_4(G)$ for both $\epsilon=1$ and $\epsilon=-1$ if and only if $\langle a, [a^b, a]\rangle$ and $\langle a^b, [a^b, a]\rangle$ are nilpotent of class at most $2$ for all $b\in G$.
\end{enumerate}
\end{cor}
\begin{cor}\label{co2}
Let $G$ be a group and $a,b\in G$ such that $[b^{\epsilon}, a^{\tau}, a^{\tau},  a^{\tau}, a^{\tau} ] = 1$ for all $\epsilon,\tau\in\{-1,+1\}$. Then
\begin{enumerate}
\item $\langle [a^b,a]\rangle^{\langle a\rangle}$ and $\langle [a^b,a]\rangle^{\langle a^b\rangle}$ are both abelian.
  \item $[a^b,a]^{a^2}=[a^b,a]^{2a-1}$
  \item $[a^b,a]^{a^{2b}}=[a^b,a]^{2a^b-1}$
  \item $[a^b,a]^{a^{-1}}=[a^b,a]^{-a+2}$
  \item $[a^b,a]^{a^{-b}}=[a^b,a]^{-a^b+2}$
  \item $[a^b,a]^{a^ba}=[a^b,a]^{-1+aa^b+1}$
  \item $[a^b,a]^{aa^ba}=[a^b,a]^{-1+2aa^b-a^b+1}$
  \item $[a^b,a]^{aa^{2b}}=[a^b,a]^{a^b+2aa^b-a-a^b}$
\end{enumerate}
\end{cor}
We  use  the following result due to Sims \cite{sim}.
\begin{thm}\label{sim}
Let $F$ be the free group of rank $2$. Then the $5$-th term $\gamma_5(F)$  of the lower central series of $F$   is equal to $N_5$  the normal closure of the basic commutators of weight $5$.
\end{thm}
\begin{rem}\label{rem}{\rm Suppose that  $a,b$ are arbitrary elements of a group and let $H=\langle a,a^b\rangle$.  Then a set of basic commutators of weight 5 on $\{a,a^b\}$ is
$\{ x_1=[a^b,a,a,a,a]$, $x_2=[a^b,a,a,a,a^b]$,
$x_3=[a^b,a,a,a^b,a^b]$, $x_4=[a^b,a,a^b,a^b,a^b]$,
$x_5=[[a^b,a,a],[a^b,a]]$, $x_6=[[a^b,a,a^b],[a^b,a]]\}$.
Hence, by Theorem $\ref{sim}$, we have $\gamma_5(H)=\langle x_1,\dots,x_6 \rangle^H$. From on now, we fix and use the notation $x_1,\dots,x_6$ as the mentioned commutators.}
\end{rem}
In the following calculation, one must be careful with notation.
As usual $u^{g_1+g_2}$ is shorthand notation for
$u^{g_1}u^{g_2}$. This means that
$u^{(g_1+g_2)(h_1+h_2)}=u^{(g_1+g_2)h_1+{(g_1+g_2)h_2}}$ which does not have
to be equal to $u^{g_1(h_1+h_2)+g_2(h_1+h_2)}$. We also have that
$$u^{(g_1+g_2)(-h)}=((u^{g_1}.u^{g_2})^{-1})^h=u^{-g_2h-g_1h}.$$
 This does not have to be the same as $u^{-g_1h-g_2h}$.
\begin{lem}\label{lm2}
Let $G$ be a group and $a,b\in G$ such that $[b^{\epsilon},
a^{\tau}, a^{\tau},  a^{\tau}, a^{\tau} ] = 1$ for all
$\epsilon,\tau\in\{-1,+1\}$. Then
 $[a^b,a,a,a^b,a^b]=[x,x^{aa^b}]^{x^{aa^b}}$, where
$x=[a^b,a]$, and $\gamma_5(\langle a,a^b\rangle)=\langle
[a^b,a,a,a^b,a^b]\rangle^{\langle a,a^b \rangle}$.
\end{lem}
\begin{proof}
 By  Lemma \ref{car}, we have $x_1=x_2=x_4=x_5=x_6=1$. Now Remark \ref{rem} completes the proof of the second part.
To prove the first part,  by Corollary \ref{co2} we may write
\begin{eqnarray}
[a^b,a,a,a^b,a^b]&=&[x^{(-1+a)},a^b,a^b]\nonumber\\
&=&x^{(-1+a)(-1+a^b)(-1+a^b)}\nonumber\\
&=&x^{-aa^b+a^b-1+a-aa^b+a^b-a^{2b}+aa^{2b}}\nonumber\\
&=&x^{-aa^b+a^b-1+a-aa^b+1+2aa^b-a-a^b}\nonumber\\
&=&x^{-aa^b-1-aa^b+1+2aa^b}\nonumber\\
&=&x^{-aa^b}[x,x^{aa^b}]x^{aa^b}\nonumber
\end{eqnarray}
\end{proof}
\noindent{\bf Proof Of Theorem \ref{th3}.}
By Corollary \ref{co2}, $\langle
[a^b,a]\rangle^{\langle a\rangle}$ is an abelian group generated
by $[a^b,a]$ and $[a^b,a]^a$.
As $[a^b,a^{-1}]=[a^b,a]^{-a^{-1}}$, the subgroup $\langle a, [a^b,a^{-1}]\rangle$ is nilpotent of class at most $2$ for all $a\in G$.   Thus by replacing $a$ by $a^{-1}$ in the latter group, we have that
$$\langle a,[a^{-b},a]\rangle=\langle
a,[a^b,a]^{-a^{-b}}\rangle=\langle a,[a^b,a]^{-a^b+2}\rangle$$ is
nilpotent of class at most 2 for every element $a\in G$. Now if
$x=[a^b,a]$, it is enough to show that $[x^{a^b},x^a]=1$.  Since $\langle a,x^{-a^b+2}\rangle$ is
nilpotent of class at most 2, we have
\begin{eqnarray}
1&=&[x^{-a^b+2},a,a]\nonumber\\
&=&x^{(-a^b+2)(-1+a)(-1+a)}\nonumber\\
&=&x^{(-2+a^b-a^ba+2a)(-1+a)}\nonumber\\
&=&x^{(-3+a^b-aa^b+1+2a)(-1+a)}\nonumber\\
&=&x^{-2a-1+aa^b-a^b+3-3a+a^ba-aa^ba+a+2a^2}\nonumber\\
&=&x^{aa^b-a^b+2-3a+a^b-aa^b+3a-2}\nonumber\\
&=&x^{-aa^b-2+aa^b+2}x^{-a^b-3a+a^b+3a}\nonumber\\
&=&[x^{aa^b},x^2][x^{a^b},x^{3a}].\nonumber
\end{eqnarray}
Therefore $[x^2,x^{aa^b}]=[x^{a^b},x^{3a}]\hspace*{1.05cm}~~(1)$.

Also
\begin{eqnarray}
1&=&[x^{-a^b+2},a,x^{-a^b+2}]\nonumber\\
&=&[x^{(-a^b+2)(-1+a)},x^{-a^b+2}]\nonumber\\
&=&[x^{(-2+a^b-a^ba+2a)},x^{-a^b+2}]\nonumber\\
&=&[x^{(-3+a^b-aa^b+1+2a)},x^{-a^b+2}]\nonumber\\
&=&[x^{-aa^b+2a},x^{-a^b+2}]\nonumber\\
&=&[x^{-aa^b},x^{-a^b+2}]^{x^{2a}}[x^{2a},x^{-a^b+2}]\nonumber\\
&=&[x^{-aa^b},x^2][x^{2a},x^{-a^b}]\nonumber\\
&=&[x^2,x^{aa^b}][x^{a^b},x^{2a}].\nonumber
\end{eqnarray}
Thus $[x^2,x^{aa^b}]=[x^{a^b},x^{2a}]^{-1}\hspace*{1.05cm}~~~(2)$.

Furthermore
\begin{eqnarray}
1&=&[x^{a^b-2},a,a]\nonumber\\
&=&x^{(a^b-2)(-1+a)(-1+a)}\nonumber\\
&=&x^{(1-a^b+aa^b+1-2a)(-1+a)}\nonumber\\
&=&x^{2a-1-aa^b+a^b-1+a-a^ba+aa^ba+a-2a^2}\nonumber\\
&=&x^{2a-1-aa^b+a^b-1+a-1-aa^b+1-1+2aa^b-a^b+1+a-4a+2}\nonumber\\
&=&x^{-aa^b+a^b-2+a+aa^b-a^b+2-a}\nonumber\\
&=&x^{-aa^b-2+aa^b+2}x^{-a^b-a+a^b+a}\nonumber\\
&=&[x^{aa^b},x^2][x^{a^b},x^a].\nonumber
\end{eqnarray}
Therefore $[x^2,x^{aa^b}]=[x^{a^b},x^a]\hspace*{1.05cm}~~(3)$.

Now by (3),  $[x^{a^b},x^a]^{x^a}=[x^{a^b},x^a]$ and by (1) and (3)
we have $[x^{a^b},x^a]^3=[x^{a^b},x^{3a}]=[x^{a^b},x^a]$. Hence
$[x^{a^b},x^a]^2=1~~(*)$. Also by (2) and (3) we have
$[x^{a^b},x^a]^{-2}=[x^{a^b},x^{2a}]^{-1}=[x^{a^b},x^a]$. Thus
$[x^{a^b},x^a]^3=1~~(**)$. Now it follows from  $(*)$ and $(**)$ that
$[x^{a^b},x^a]=1$. This completes the proof. $\hfill \Box$
\section{\bf Left 4-Engel elements}
In this section we prove Theorems \ref{th4} and \ref{th5}.
The argument of Lemma \ref{lm5} is very much modeled on an argument given
in \cite{traus}.
\begin{lem}\label{lm5}
Let $G$ be any group. If $a^{\pm 1}\in L_4(G)$, then $[x^{a^b},x^a]=[x,x^{aa^b}]$, where $x=[a^b,a]$ for all $b\in G$.
\end{lem}
\begin{proof}
From  Corollary \ref{co1}-(b) we have $\langle a,[a^b,a]\rangle \in \mathcal{N}_2$
for all $b\in G$.
Thus $\langle a,[a^{bx},a]\rangle \in \mathcal{N}_2$. Therefore $[a^{bx},a^{-1},a,a]=1$.
 We have
\begin{eqnarray}
[a^{bx},a]&=&[x^{-1}a^bx,a]\nonumber\\
&=&[x^{-1}a^b,a]^x[x,a]\nonumber\\
&=&[x^{-1},a]^{a^bx}[a^b,a]x^{-1}x^a\nonumber\\
&=&(xx^{-a})^{a^bx}x^a\nonumber\\
&=&x^{-1-aa^b+a^b+1+a}.\nonumber
\end{eqnarray}
Let $y=[a^{bx},a^{-1}]=[a^{bx},a]^{-a^{-1}}$. Then
\begin{eqnarray}
y&=&x^{-1-a^{-1}-a^ba^{-1}+aa^ba^{-1}+a^{-1}}\nonumber\\
&=&x^{-1-a^{-1}+(a^{-1}-a^{-1}a^b-a^{-1})+(a^{-1}+a^b-a^{-1})+a^{-1}}\nonumber\\
&=&x^{-1+aa^b-a^b}.\nonumber
\end{eqnarray}
We then have $y^{-a}=[a^{bx},a]=x^{-1-aa^b+a^b+1+a}$ and
\begin{eqnarray}
y^{a^2}&=&x^{-a^2-a-a^ba+aa^ba+a}\nonumber\\
&=&x^{-3a+aa^b-a^b+1+a}.\nonumber
\end{eqnarray}
Therefore
\begin{eqnarray}
1&=&y^{-a}yy^{-a}y^{a^2}\nonumber\\
&=&x^{-1-aa^b+a^b+1+a}x^{-1+aa^b-a^b}x^{-1-aa^b+a^b+1+a}x^{-3a+aa^b-a^b+1+a}\nonumber\\
&=&x^{-1+a^b+a-a^b}x^{-1-aa^b+a^b+1-2a+aa^b-a^b+1+a}\nonumber\\
&=&x^{-1+a^b-1-aa^b+1-a+aa^b-a^b+1+a}\nonumber\\
&=&x^{-1+a^b}x^{-1-aa^b+1+aa^b}x^{-a-a^b+a+a^b}x^{-a^b+1}\nonumber\\
&=&x^{-1+a^b}[x,x^{aa^b}][x^a,x^{a^b}]x^{-a^b+1}.\nonumber
\end{eqnarray}
Conjugation with $x^{-1+a^b}$ gives
$$1=[x,x^{aa^b}][x^a,x^{a^b}].$$
\end{proof}
\noindent{\bf Proof of Theorem \ref{th4}}. By Lemma \ref{lm2} we have to prove that
$x_3=1$.
 Let $u=[x^{a^b},x^a]=[x,x^{aa^b}]$. By Lemma \ref{lm2}
and Lemma \ref{lm5} we have $x_3=u^{x^{aa^b}}=u$.  Thus
$$\gamma_5(\langle a,a^b\rangle)=\langle u\rangle^{\langle a,a^b\rangle}.$$
Since
\begin{eqnarray}
u^a&=&[x^a,x^{aa^ba}]\nonumber\\
&=&[x^a,x^{-1+2aa^b-a^b+1}]\nonumber\\
&=&[x^a,x^{-a^b}]\nonumber\\
&=&u\nonumber
\end{eqnarray}
and
\begin{eqnarray}
u^{a^b}&=&[x^{a^{2b}},x^{aa^b}]\nonumber\\
&=&[x^{2a^b-1},x^{aa^b}]\nonumber\\
&=&[x^{-1},x^{aa^b}]\nonumber\\
&=&u^{-1}\nonumber
\end{eqnarray}
we have $\gamma_5(\langle a,a^b\rangle)=\langle u\rangle$ and
\begin{eqnarray}
\gamma_6(\langle a,a^b\rangle)&=&[\langle a,a^b\rangle,
\gamma_5(\langle a,a^b\rangle)]\nonumber\\
&=&[\langle a,a^b\rangle,\langle u\rangle]\nonumber\\
&=&\langle u^2\rangle. \nonumber
\end{eqnarray}
Also
$1=[u,a^b,a^b]=[u^{-2},a^b]$ and we have $\gamma_7(\langle
a,a^b\rangle)=1$. On the other hand
\begin{eqnarray}
[[a^b,a,a,a^b],[a^b,a]]&=&[x^{-a+1-a^b+aa^b},x]\nonumber\\
&=&[x^{aa^b},x]\nonumber\\
&=&u^{-1}\in \gamma_6(\langle a,a^b\rangle).\nonumber
\end{eqnarray}
Thus $u=1$ and this completes the proof. $\hfill \Box$
\begin{cor}\label{co3}
Let $G$ be a group and $a^{\pm 1}\in L_4(G)$. Then $\langle
a,a^b\rangle'$ is abelian, for all $b \in G$.
\end{cor}
\begin{cor}\label{co4}
Let $G$ be an arbitrary group and $a^{\pm 1}\in L_4(G)$. Then
every power of $a$ is also a left $4$-Engel element.
\end{cor}
\begin{proof}
By Corollary \ref{co2}, Theorem \ref{th4} and Corollary \ref{co3},
$$\langle a^b,[a^b,a],[a^b,a]^a\rangle=\langle
a^b,[a^b,a],[a^b,a,a]\rangle\in \mathcal{N}_2$$  for all $b\in G$. It follows that $\langle a^{ib},[a^{ib},a^i]\rangle \leq
\langle a^b,[a^b,a],[a^b,a]^a\rangle$ and $\langle
a^{ib},[a^{ib},a^i]\rangle\in\mathcal{N}_2$ for all $b\in G$ and
$i\in \Z$.  Now Corollary \ref{co1} implies that  $a^i\in L_4(G)$ for
all $i\in \Z$.
\end{proof}
\begin{lem}
Let $G$ be a group and $a^{\pm 1}\in L_4(G)$. Then for all $b$ in
$G$ and $r,m,n\in \N$ we have
\begin{enumerate}\label{lm6}
    \item $[a^b,a^r]^{a^{2m}+1}=[a^b,a^r]^{2a^m}$
    \item $[a^{nb},a^r]^{a^{2mb}+1}=[a^{nb},a^r]^{2a^{mb}}$
    \item If $a^s=1$ then $[a^b,a,a^s]=[a^b,a,a]^s$
    \item $[a^b,a]^{a^n}=[a^b,a]^{na-(n-1)}$
    \item $[a^b,a]^{a^{nb}}=[a^b,a]^{na^b-(n-1)}$
\end{enumerate}
\end{lem}
\begin{proof}
By Corollary \ref{co2}, $[a^b,a]^{\langle a\rangle}$ and
$[a^b,a]^{\langle a^b\rangle}$ are both abelian and $$\langle
a^{mb},[a^{nb},a^r]\rangle\leq \langle
a^b,[a^b,a],[a^b,a]^a\rangle$$ for all $b$ in $G$. Therefore both
$\langle a^m,[a^b,a^r]\rangle$ and $ \langle
a^{mb},[a^{nb},a^r]\rangle$ are nilpotent of class at most 2, for
all $b\in G$ and $r,m,n\in \N$. Thus
\begin{eqnarray}
1=[a^b,a^r,a^m,a^m]&=&[a^b,a^r]^{(-1+a^m)^2}\nonumber\\
&=&[a^b,a^r]^{1-2a^m+a^{2m}}.\nonumber
\end{eqnarray}
This prove part (1). Part (2) is similar to part (1) and the other
parts are straightforward by Corollary \ref{co2} and induction.
\end{proof}
\begin{lem}\label{lm7}
Let $G$ be a group and $a^{\pm 1}\in L_4(G)$. If $o(a)=p^i$, for
$p=2$ and $i\geq 3$ or some odd prime $p$ and $i\geq 2$, then
$a^{p^{i-1}}\in L_2(G)$.
\end{lem}
\begin{proof}
First let $p=2$ and $m=p^{i-3}$. Then $a^{8m}=1$ and we have so
\begin{eqnarray}
1=[a^b,a^{8m}]&=&[a^b,a^{4m}]^{1+a^{4m}}\nonumber\\
&=&[a^b,a^{4m}]^2.~~\hspace*{1.05cm}~~~ \text{by Lemma}~
\ref{lm6}\nonumber
\end{eqnarray}
Now we have
$$[b,a^{p^{i-1}},a^{p^{i-1}}]=[b,a^{4m},a^{4m}]=[a^{4mb},a^{4m}]^{-a^{4m(-b+1)}}.$$
But
\begin{eqnarray}
[a^{4mb},a^{4m}]&=&[a^{2mb},a^{4m}]^{a^{2mb}+1}\nonumber\\
&=&[a^{2mb},a^{4m}]^{2a^{mb}} \hspace*{1.05cm}~~~\text{by Lemma}\; \ref{lm6}\nonumber\\
&=&[a^b,a^{4m}]^{2(a^{2m-1}+\cdots+a^b+1)a^{mb}}\nonumber\\
&=&1.\nonumber
\end{eqnarray}
This complete the proof of the lemma in this case. Now let $p$ be
an odd prime number and $i\geq 2$. Then we have
\begin{eqnarray}
1&=&[a^b,a^{p^i}]\nonumber\\
&=&[a^b,a]^{1+a+\cdots+a^{p^i-1}}\nonumber\\
&=&[a^b,a]^{1+a+2a-1+\cdots+(p^i-1)a-(p^i-2)} \hspace*{1.05cm} \text{by Lemma}~ \ref{lm6}\nonumber\\
&=&[a^b,a]^{\frac{p^i(p^i-1)}{2}a-\frac{p^i(p^i-1)}{2}+p^i}\nonumber\\
&=&[a^b,a]^{(p^ia-p^i)\frac{p^i-1}{2}}[a^b,a]^{p^i}\nonumber
\end{eqnarray}
Now since
$$1=[a^b,a,a^{p^i}]=[a^b,a,a]^{p^i}=[a^b,a]^{p^i a-p^i}
\hspace*{.5cm} \text{by Lemma  \ref{lm6} part (3)}$$ we have
$[a^b,a^{p^i}]=[a^b,a]^{p^i}=1$. Let $m=p^{i-1}$. Then
$$[b,a^m,a^m]=[a^{mb},a^m]^{-a^{-mb+m}}$$
and by Corollary \ref{co3} and Lemma \ref{lm6} we have so
\begin{eqnarray}
[a^{mb},a^m]&=&[a^b,a^m]^{(1+a^b+\cdots+a^{(m-1)b})}\nonumber\\
&=&[a^b,a^{m}]^{(1+a^b+2a^b-1+\cdots+(m-1)a^b-(m-2))}\nonumber\\
&=&[a^b,a]^{(1+a+2a-1+\cdots+(m-1)a-(m-2))(1+a^b+2a^b-1+\cdots+(m-1)a^b-(m-2))}
~\nonumber\\
&=&[a^b,a]^{(\frac{m(m-1)}{2}a-\frac{m(m-1)}{2}+m)(\frac{m(m-1)}{2}a^b-\frac{m(m-1)}{2}+m)}\nonumber\\
&=&[a^b,a]^{m^2(\frac{m-1}{2}a-\frac{m-3}{2})(\frac{m-1}{2}a^b-\frac{m-3}{2})}.\nonumber
\end{eqnarray}
Now since $p^i\mid m^2$ and $[a^b,a]^{p^i}=1$ we have
$$[b,a^{p^{i-1}},a^{p^{i-1}}]=[b,a^m,a^m]=1.$$
This complete the proof of the lemma.
\end{proof}
\noindent{\bf Proof of Theorem \ref{th5}.}
 Let $p$ be a prime number and $o(a)=p^i$.
If $p=2$ and $i\leq 2$ then the assertion is obvious. Therefore
let $i\geq 3$ if $p=2$; and $i\geq 2$ if $p$ is an odd prime
number. By Lemma \ref{lm7} and Corollary \ref{co1}
$$1\trianglelefteq \langle a^{p^{n-1}}\rangle^G \trianglelefteq
\langle a^{p^{n-2}}\rangle^G\trianglelefteq \cdots\trianglelefteq
\langle a^p\rangle^G$$ is a series of normal subgroups of $G$ with
abelian factors. This implies that $K=\langle a^p\rangle^G$ is
soluble of derived length at most $n-1$. By Corollary \ref{co4},
$a^p$ and so all its conjugates in $G$ belong to $L_4(G)$ and in
particular they are in $L_4(K)$. Now a result of Gruenberg
\cite[Theorem 7.35]{robin} implies that $B(K)=K$. Therefore
$a^p\in K\leq B(G)$, as required. $\hfill \Box$

\section{\bf Examples and questions}
In this section we give some examples by using {\sf GAP nq} package of Werner Nickel to show what we mentioned in
the last two paragraphs of Section 1.\\

  Let $H$ be the largest nilpotent group
generated by $a,b$ such that $a\in R_4(H)$; then $H$ is nilpotent
of class 8. On the other hand if $K$ is the largest nilpotent
group generated by $a$ and $b$ such that $a^{\pm 1}\in
R_4(K)$, then $K$ is nilpotent of class 7. Thus in an arbitrary
group $G$, $a\in R_4(G)$ does not imply $a^{-1}\in R_4(G)$. One
can check the above argument with the following GAP program:
\begin{verbatim}
LoadPackage("nq"); #nq package of Werner Nickel#
F:=FreeGroup(3);; a:=F.1;; b:=F.2;; x:=F.3;;
G:=F/[LeftNormedComm([a,x,x,x,x])];; H:=NilpotentQuotient(G,[x]);;
NilpotencyClassOfGroup(H);
G:=F/[LeftNormedComm([a,x,x,x,x]),LeftNormedComm([a^-1,x,x,x,x])];;
K:=NilpotentQuotient(G,[x]);; NilpotencyClassOfGroup(K);
\end{verbatim}
Similar to above let $N$ be the largest nilpotent group generated
by $a,b$ such that $a,b\in L_4(N)$ and $b^2=1$. Also let
$S=\langle a,a^b\rangle$. Then $N$ is nilpotent of class 10 and
$S$ is nilpotent of class 6 but the largest nilpotent group $M$
generated by $a,b$ such that $a^{\pm 1},b\in L_4(M)$ and $b^2=1$
is nilpotent of class 7. Therefore in an arbitrary group $G$, $a\in L_4(G)$ does not
imply $a^{-1}\in L_4(G)$ and $b^{-1}\in L_4(G)$ is a necessary
condition in Theorem \ref{th4} and \ref{th5}. The following GAP
program confirms the above argument:
\begin{verbatim}
F:=FreeGroup(3);; a:=F.1;; b:=F.2;; x:=F.3;;
G:=F/[LeftNormedComm([x,a,a,a,a]),LeftNormedComm([x,b,b,b,b]),b^2];;
N:=NilpotentQuotient(G,[x]);; NilpotencyClassOfGroup(N);
S:=Subgroup(N,[N.1,N.2^-1*N.1*N.2]);; NilpotencyClassOfGroup(S);
G:=F/[LeftNormedComm([x,a,a,a,a]),LeftNormedComm([x,b,b,b,b]),
LeftNormedComm([x,a^-1,a^-1,a^-1,a^-1]), b^2];;
M:=NilpotentQuotient(G,[x]);; NilpotencyClassOfGroup(M);
\end{verbatim}

We end this section by proposing some questions on bounded right Engel elements in certain classes of groups.
\begin{qu}\label{ql1}
Let $n$ be a positive integer. Is there a set of prime numbers $\pi_n$ depending only on $n$ and a function $f:\mathbb{N}\rightarrow \mathbb{N}$ such that the nilpotency class of $\langle x \rangle ^G$ is at most $f(n)$ for any $\pi_n'$-element $x\in R_n(G)$ and any nilpotent or finite group $G$?
\end{qu}
 \begin{qu}\label{ql2}
 Let $n$ be a positive integer. Is there a set of prime numbers $\pi_n$ depending only on $n$ such that the set of right $n$-Engel elements in any nilpotent or finite  $\pi'_n$-group forms a subgroup?
 \end{qu}
 Note that if the answers of Questions \ref{ql1} and \ref{ql2} are positive, the answers of the corresponding questions for residually (finite or nilpotent $\pi'_n$-groups) are also positive.
 \begin{qu}
 Let $n$ and $d$ be positive integers. Is there a function $g:\mathbb{N}\times \mathbb{N}\rightarrow \mathbb{N}$ such that any nilpotent group generated by $d$ right $n$-Engel elements is nilpotent of class at most $g(n,d)$?
 \end{qu}

{\noindent\bf Acknowledgements.} The authors are grateful to the referee for his very helpful comments. They also thank Yaov Segev for pointing out a flaw in our proofs.
The research of the first author was supported in part by the Center of Excellence for Mathematics, University of Isfahan.

\end{document}